\documentclass{amsart}

\usepackage[margin=3cm,vmargin=3cm]{geometry}
\newtheorem{theorem}{Theorem}[section]
\newtheorem{lemma}[theorem]{Lemma}
\newtheorem{proposition}[theorem]{Proposition}
\newtheorem{corollary}[theorem]{Corollary}

\theoremstyle{definition}
\newtheorem{definition}[theorem]{Definition}
\newtheorem{example}[theorem]{Example}
\newtheorem{remark}[theorem]{Remark}

\usepackage{url,graphicx,amssymb,enumerate,stmaryrd, amsthm}
\usepackage[pagebackref,colorlinks,citecolor=blue,linkcolor=blue,urlcolor=blue,filecolor=blue]{hyperref}
\usepackage{microtype}
\usepackage[all]{xy}
\usepackage{dsfont}
\usepackage{aliascnt}
\usepackage[vcentermath]{youngtab}
\usepackage[dvipsnames, hyperref]{xcolor}
\usepackage{pinlabel}
\usepackage{etoolbox}
\usepackage{tikz} \usetikzlibrary{matrix,arrows}

\usepackage{aliascnt}

\numberwithin{thmcounter}{section}
\newaliascnt{thmauto}{thmcounter}

\newaliascnt{Propauto}{thmcounter}

\newaliascnt{Defauto}{thmcounter}

\newaliascnt{exauto}{thmcounter}

\newaliascnt{quesauto}{thmcounter}

\newaliascnt{goalauto}{thmcounter}

\newaliascnt{conjauto}{thmcounter}

\newcommand{\cA}{\mathcal{A}}
\newcommand{\cC}{\mathcal{C}}

\newcommand{\hS}{\widehat{\rm S}}

\newcommand{\cU}{\mathcal{U}}

\newcommand{\C}{\mathbb{C}}

\newcommand{\im}{\mathrm{im}}
\newcommand{\id}{\mathrm{id}}

\newcommand{\I}{{\bf I}}

\newcommand{\Lie}{\mathrm{Lie}}

\newcommand{\cs}{cs}

\newcommand{\B}{\widehat{{\rm PBr}}}
\newcommand{\hBr}{\widehat{\rm Br}}
\newcommand{\PBr}{{\rm PBr}}
\newcommand{\Br}{{\rm Br}}
\newcommand{\bC}{\mathbb{C}}

\newcommand{\bL}{\mathbb{L}}

\newcommand{\bS}{\mathbf{S}}

\newcommand{\bbF}{\mathbb{F}}\newcommand{\bbZ}{\mathbb{Z}}

\newtheorem{atheorem}{Theorem}

\newcommand{\rS}{{\rm S}}

\newcommand{\Sym}{\mathbb{\mr{Sym}}}
\newcommand{\Conf}{\mathrm{Conf}}

\newcommand{\into}{\hookrightarrow}
\newcommand{\onto}{\twoheadrightarrow}
\newcommand{\FI}{\mathrm{FI}}
\newcommand{\hFI}{\widehat{\mathrm{FI}}}

\newcommand{\sgn}{\mathrm{sgn}}

\newcommand{\iso}{\cong}

\newcommand{\R}{\mathbb{R}}
\newcommand{\Q}{\mathbb{Q}}
\newcommand{\Z}{\mathbb{Z}}
\newcommand{\m}{\to}

\newcommand{\mr}[1]{{\rm #1}}

\newcommand{\coker}{\mr{coker}}
\newcommand{\Hom}{\mathrm{Hom}}

\newcommand{\Mod}{\mathrm{Mod}}
\newcommand{\Rep}{\mathrm{Rep}}

\newcommand{\Tor}{\mathrm{Tor}}
\newcommand{\Ind}{\mathrm{Ind}}

\DeclareMathOperator{\rA}{A}

\newcommand{\inv}{^{-1}}
\newcommand{\rB}{ B}

\title{Representation stability for pure braid group Milnor fibers}
\date{\today}

\author{Jeremy Miller}
\address{Department of Mathematics, Purdue University, West Lafayette, IN}
\email{\href{mailto:jeremykmiller@purdue.edu}{jeremykmiller@purdue.edu}}
\urladdr{\url{https://www.math.purdue.edu/~mill1834/}}
\thanks{Jeremy Miller was supported in part by NSF grant DMS-1709726.}

\author{Philip Tosteson}
\address{Department of Mathematics, University of Chicago, Chicago, IL}
\email{\href{mailto:ptoste@math.uchicago.edu}{ptoste@math.uchicago.edu}}
\urladdr{\url{http://www-personal.umich.edu/~ptoste/}}

\begin{document}

\begin{abstract}  
We prove a representation stability result for the Milnor fiber associated to the pure braid group. Our result connects previous work of Simona Settepenella to representation stability in the sense of Church--Ellenberg--Farb, answering a question of Graham Denham. We also use our result to compute the stable integral homology of the Milnor fiber.   
\end{abstract}

\maketitle
\tableofcontents

\date{\today}





\maketitle
\section{Introduction}
\label{sec:intro}

Let $\Conf_n(\C)$ denote the configuration space of $n$ points in the complex plane. This configuration space is a hyperplane complement and we will study the homology of the associated Milnor fiber \[F_n=\left\{(x_1,\ldots, x_n) ~ |  ~\prod_{i <j}( x_i-x_j)=1  \right \} \subseteq \Conf_n(\C). \] The Milnor fiber $F_n$  admits two natural group actions:  the alternating group $\rA_n$  acts by permuting the coordinates,  and the ${n \choose 2}$th  roots of unity act by multiplying the coordinates.   In fact, we can extend these actions to the action of a single group  $$\hS_n  :=  \{(\sigma, d) \in \rS_n \times \mathbb Z ~|~ d \text{ odd }  \iff {\rm sgn } \sigma = {-1}\},  ~ n \geq 2. $$  The element  $(\sigma, d)$ acts by    $$(x_1, \ldots, x_n) \mapsto  (\zeta_{n(n-1)}^d~ x_{\sigma(1)},\ldots,\zeta_{n(n-1)}^d~ x_{\sigma(n)}),  $$  where   $\zeta_k := \exp(\frac{2 \pi i }{k})$ is the distinguished primitive $k$th root of unity.  

In \cite[Theorem 1.1]{Sett}, Settepanella showed that, for $n \geq 3k-2$,  the action by roots of unity on rational homology is trivial. This result prompted Denham to ask whether the homology of the Milnor fiber exhibits any form of representation stability  \cite[Problem 10]{MFOgram}. In this paper, we establish a representation stability result for the homology of $F_n$ which incorporates the action of $\hS_n$.

 
To state this theorem,  we use a category $\hFI$ which is built out of quotients of  $\hS_n$  in the same way as the category $\FI$, of finite sets and injections, is built from $\rS_n$.  
We define an action of $\hFI$ on $H_*(F_n, \bbZ)$,  and prove a finite generation result.

\begin{atheorem} \label{theoremA}
For all $i$, the sequence $\{H_i(F_n,\Z)\}_n$ is a finitely generated $\hFI$-module. 
\end{atheorem}

See Definition \ref{defGen} for a definition of finite generation and see Theorem \ref{mainRange} for a version of Theorem \ref{theoremA} with explicit stability bounds.

This theorem has several consequences. We show that if $M_n$ is a finitely generated $\hFI$-module, then  for $n$ sufficiently large, the subgroup $\bbZ \subset \hS_n$ acts trivially and hence $M_n$ agrees with a finitely generated $\FI$-module in a stable range.  In this way, Theorem \ref{theoremA} incorporates features of  both the phenomenon Settepanella established and representation stability for symmetric group representations.


\begin{atheorem} \label{theoremB}
For all $n \geq 2+8i+3 i^2$, the roots of unity $\mu_{n \choose 2}$ act trivially on $H_i(F_n, \bbZ)$. In this range, $H_i(F_n, \bbZ)$ agrees with a finitely generated $\FI$-module. 
 
\end{atheorem}

In particular, the rational  $S_n$ representations $H_i(F_n, \Q)$ exhibit representation stability in the sense of Church--Farb \cite[Definition 1.1]{CF} (see Church--Ellenberg--Farb \cite[Theorem 1.13]{CEF}). In \cite[Theorem 1.2]{Sett}, Settepanella computed the groups $H_i(F_n, \Q)$ in a stable range. Using our results, we are able to extend this to an integral calculation. 


\begin{atheorem} \label{theoremC} For $n \geq 2+8i+ei^2$, there is an $\rS_n$-equivariant injection: \[H_i(F_n, \Z) \to H_i(\Conf_n(\C)/\C^*, \Z).\]  The cokernel agrees with a finitely generated $\FI$-module consisting of torsion abelian groups. 
\end{atheorem}

In particular, the group $H_i(F_n, \Z)$ is noncanonically isomorphic to $H_i(\Conf_n(\C)/\C^*, \Z)$ in a stable range. The homology of $\Conf_n(\C)/\C^*$ is canonically isomorphic to the homology of the moduli space of genus $0$ curves with $n+1$ marked points, $M_{0,n+1}$.  The $\rS_n$ representation $H^i(M_{0,n+1})$ has been calculated by Getzler \cite[Theorem 5.7]{GetzlerModuli}.  In the appendix, we give a self-contained description of the homology.

Our method of proof of Theorem \ref{theoremA} involves considering highly connected semi-simplicial sets with actions of  the groups $\hS_n$ and $\pi_1(F_n)$. This is an adaptation of Quillen's approach to proving homological stability. The proof is in the spirit of Putman \cite{Pu} and largely fits into the axiomatic framework of Patzt \cite{Pa2}.

Similar theorems are likely true for the Milnor fibers associated to the type $B$ and type $D$ braid groups.  Additionally, we expect that the techniques of this paper apply to prove representation stability for homology of the subgroup of surface braid groups with total winding number zero.  We will not consider these generalizations here.

\subsection{Description of Stabilization Maps}
The category $\hFI$ does not act naturally on the Milnor fiber $F_n$;  we only construct an action of $\hFI$ up to homotopy.   Our situation is analogous to the action of $\FI$ on $H_*(\Conf_n(\C))$ by adding points, where $\FI$ only acts on $\Conf_n(\C)$ up to homotopy.  For this $\FI$ action, a  representative of the standard injection $[n] \into [n+1]$ is given by a map $\Conf_n(\C) \to \Conf_{n+1}(\C)$  that adds an $n+1$st point to the right of the first $n$ points.  

Our stabilization map $F_n \to F_{n+1}$ is induced by the map $\Conf_n(\C) \to \Conf_{n+1}(\C)$ in the following sense.  The  Milnor fiber $F_n$ is a $K(\pi,1)$, and we have  $\pi_1(F_n) \subset \pi_1( \Conf_n(\C))$. The $\FI$ stabilization map $\Conf_n(\C) \to \Conf_{n+1}(\C)$ takes $\pi_1( F_n) \to \pi_1( F_{n+1})$,  and we  define the action of $\hFI$  so that $[e] \in \hS_{n+1}/\hS_n = \hFI(n,n+1)$ acts on $\pi_1$ by this inclusion.  This suffices to determine the action $e: F_n \to F_{n+1}$  up to homotopy.  

We have two other, more geometric, descriptions of this stabilization map.  To describe the first,  we replace the Milnor fiber $F_n$  by the covering space $F_n'$ of  $\Conf_n(\C)$  associated to the inclusion of fundamental groups $\pi_1 (F_n) \into \pi_1( \Conf_n(\C))$.  This space can be described as $$F_n' = \left\{(x_i)_{i = 1}^n \in \Conf_n ( \mathbb C),  z \in \mathbb C~|~  \prod_{i < j}(x_i - x_j) = \exp(z)  \right\}, $$  since taking $\log$ shows that this is a cover and there is a deformation retraction of the map $F_n' \to \Conf_n (\mathbb C)$ to $F_n \to \Conf_n (\mathbb C)$, given by taking $z \to \lambda z$  and $x_i \to x_i \exp(\lambda /{n \choose 2} )  $ for $\lambda \in [0,1]$.
There is a unique lift of any choice of stabilization map $\Conf_n(\C) \to \Conf_{n+1}(\C)$ to a map of covers $F_n' \to F_{n+1}'$.  On homology, this induces the action of $[e] \in \hFI(n,n+1)$.

Second, Gadish has described a stabilization map on the Milnor fiber $F_n$ itself, which induces the action of $[e] \in \hFI(n,n+1)$ on homology.  Gadish's observation is that given a configuration  $(x_i)_{i = 1}^n$  such that $\prod_{1 \leq  i < j \leq n} (x_j - x_i) = 1$,  if  we add a point  $x_{n+1} \in  \R \subset \C$  such that $x_{n+1}$ is $\gg 0$,  then the complex  number  $$ a = \prod_{ 1\leq  i < j \leq  n+1} (x_j - x_i) = \prod_{i = 1}^n (x_{n+1}- x_i) $$ has argument $> 0$ and we can choose a branch of the function that takes ${n+1 \choose 2}$th  roots, and divide each $x_i$ by $a^{1/{n+ 1\choose 2}}$ to continuously obtain a point in $F_{n+1}$.  

To formally define Gadish's map  $e: F_n \to F_{n+1}$,  we fix a branch of the ${n+1 \choose 2}$th root function  with branch cut along the negative real axis.  Then we define  $e(x_1, \dots, x_n)$ to be $(y_1, \dots, y_{n+1})$   where $y_i = x_i/a^{1/{n+ 1\choose 2}}$ for $i \leq n$,  $a =  \prod_{i = 1}^n (x_{n+1}- x_i)$,  and $$x_{n+1} =1 + \sum_{i = 1}^n {\rm re}(x_i) +  \sum_{i = 1}^n  \frac{\im(x_i)}{{\rm tan}(\pi/n)}. $$

\subsection*{Acknowledgments}

We would like to thank Christin Bibby, Graham Denham, and Dan Petersen for helpful conversations. We thank Nir Gadish for allowing us to include his description of the stabilization map in this paper.  We thank Graham Denham, Giovanni Gaiffi, Rita Jim\'enez Rolland, and Alexander Suciu for organizing the MFO Oberwolfach program ``Topology of Arrangements and Representation Stability'' and MFO Oberwolfach for hosting. In particular, we thank Graham Denham for posing this question in the problem session.


\section{Algebraic preliminaries}

In this section, we define $\hFI$. We recall some facts concerning the theory of $\FI$-modules and describe their  implications for $\hFI$-modules.

\subsection{$\hFI$-modules} 

We begin by constructing a monoidal structure on the groupoid $$\hS:=\bigsqcup_n \hS_n.$$ Here we define  $\hS_0$ and $\hS_1$  to be trivial groups. The monoidal structure is given by the maps  $$m_{n_1, n_2}: \hS_{n_1} \times \hS_{n_2} \to \hS_{n_1 + n_2}, ~~(\sigma_1, d_1) \times (\sigma_2, d_2) \mapsto (\sigma_1 \sigma_2, d_1 + d_2).$$  When $n_1,n_2$ are clear from context, we  write $m = m_{n_1, n_2}$.    We will write  $i_1(\hS_{n_1}) \subset \hS_{n_1 + n_2}$ for $m(\hS_{n_1} \times e)$,  and $i_2(\hS_{n_2})$ for $m(e \times \hS_{n_2})$.  Since the subgroup $i_r(\hS_{n_r})$ is isomorphic to $\hS_{n_r}$,  when it is clear which embedding we are taking, we suppress $i_r$  from our notation.

The category $\hS$ has a braided monoidal structure induced by the surjection $p_n:  \Br_n \to  \hS_n$.  More precisely, the braid $\sigma_{a,b}$  which braids the first $a$ strands over the last $b$ strands conjugates $m_{a,b}$ to $m_{b,a}$. See \S\ref{conventions} for the definition of $p_n$ and our conventions on braid groups.   Since the maps $i_{r}: \hS_{n_r} \to \hS_{n_1 + n_2}$ are inclusions, the construction of Randall-Williams Wahl  \cite[Theorem 1.10]{RWW} applies to produce a monoidal category $\hFI = \cU \hS$.  We will make the definition of $\hFI$ and its monoidal structure explicit.

The category $\hFI$  has objects indexed by natural numbers,  and morphisms given by the right cosets  $$\hFI(n,m) =  \hS_m/ i_2(\hS_{m-n}) .$$    The multiplication  $\hFI(n,m) \times \hFI(m,l) \to \hFI(n,l)$  is given by $[s] \times [t] \mapsto  [t i_1(s) ]$. It is well-defined because elements of $i_2(\hS_{l-m})$ commute with $i_1(a)$,  and is associative  because  $u i_1(t i_1(s)) =  u i_1(t) i_1(s)$.  

The monoidal structure is given on objects by $n_1 \times n_2 \mapsto n_1 + n_2$,  and on morphisms by $$\hFI(n_1, m_1) \times \hFI(n_2, m_2) \mapsto \hFI(n_1 + n_2, m_1 + m_2),  [s] \times [t] \mapsto [i_1(f) i_2(g)] \circ [\tau_{m_1 - n_1, n_2} ],$$
where $\tau_{m_1 - n_1, n_2}$ denotes the element of $\Br_{m_1 + m_2}$ defined as follows. Writing \[[m_1 + m_2] = [n_1] \sqcup [m_1 - n_1] \sqcup [n_2] \sqcup[m_2 - n_2],\] we let $\tau_{m_1 - n_1, n_2}$ be the element which braids the strands of $[m_2- n_2]$ over the strands of $[n_1]$.

\begin{remark}
To obtain the monoidal category $\hFI$ as we have defined it from \cite[Theorem 1.10]{RWW},  apply their construction to the braided monoidal groupoid defined by: $$\hS_a \times \hS_b \overset{{\rm switch}}{\longrightarrow} \hS_b \times \hS_a \overset{m_{b,a}}{\longrightarrow} \hS_{a+b}.$$  
\end{remark}


Given a category $\mathcal C$, the term $\mathcal C$-module will mean a functor from the category $\mathcal C$ to the category of abelian groups. Let $\Mod_{\mathcal C}$ denote the category of $\mathcal C$-modules. For an $\hFI$-module or $\hS$-module $M$ and $n$ a natural number, let $M_n$ denote the value of $M$ on $n$. There is a functor from $\hS$ to $\hFI$ which identifies $\hS$ with the largest subcategory of $\hFI$ such that every morphism is invertible. This gives a forgetful functor $\Mod_{\hFI} \m \Mod_{\hS}$.


\begin{definition} \label{defGen}
Let $\I: \Mod_{\hS} \m \Mod_{\hFI} $ be the left adjoint to the forgetful functor. An $\hFI$-module $M$ is called \emph{induced} if $M \cong \I(W)$ for some $\hS$-module $W$. We say $M$ has generation degree $\leq d$ if there is a short exact sequence of $\hFI$-modules: $$ \I(W) \m M \m 0$$ with $W_n \cong 0$ for $n>d$. We say $M$ is finitely generated if there is a short exact sequence of $\hFI$-modules: $$ \I(W) \m M \m 0$$ with $\bigoplus_n W_n$ a finitely generated abelian group. We say $M$ has presentation degree $\leq r$ if there is a short exact sequence of $\hFI$-modules: $$ \I(V) \m \I(W) \m M \m 0$$ with $W_n \cong V_n \cong 0$ for $n>r$. 
\end{definition}

Note that if each $M_n$ is finitely generated as an abelian group, then $M$ is finitely generated if and only if it has finite generation degree. Many definitions appearing in this paper, including the above definitions, are adaptations of definitions for $\FI$-modules which have appeared in other papers. For the sake of brevity, we will often only state definitions for $\hFI$-modules but will often also use the corresponding definition for $\FI$-modules. 


\subsection{Central stability homology and regularity}\label{CSHsection}

Central stability homology is a construction which often appears on $E^2$-pages of spectral sequences used to establish representation stability. When certain semi-simplicial sets are highly connected, central stability homology controls degrees of higher syzygies \cite[Theorem 5.7]{Pa2}.

\begin{definition}\label{CentralStabilityDef}
Let $M$ be an $\hFI$-module and $n$ a natural number. For $p \geq -1$, let \[C^{\cs,\hFI}_{p}(M)_n=\Ind_{i_1(\hS_{n-(p+1)})}^{\hS_n} M_{n-(p+1)}.\] These groups assemble to form an augmented semi-simplicial $\hFI$-module,  defined in terms of the following maps.

The $\hFI$-module structure of $M$  gives maps  $x_n : \hFI(n,n+1) \otimes M_n  \to M_{n+1}$.  The automorphism group $\hFI(n,n) = \hS_n$ acts on $M_n$ on the left  and on $\hFI(n,n+1)$ on the right.  And the map $x_n$ factors the quotient to yield $$x_n : \Ind^{\hS_{n+1}}_{i_1(\hS_{n})} M_n = \hFI(n,n+1) \otimes_{\hS_n} M_n \to M_{n+1}.$$  For a braid $b \in \Br_m$,   right multiplication by $b$  gives an automorphism $\bbZ \hS_{n+m} \to \bbZ \hS_{n+m}$ as an $\hS_{n+m} , \hS_n$ bi-module. There is an induced automorphism of $\Ind_{\hS_n}^{\hS_{n+m}} M_n$, which we will also denote $b$.  Let $u_i \in \Br_{p+1}$ be the element that braids the $i$th strand over the all the others to the left,  $u_i := \sigma_{i-1, i}^{-1} \sigma_{i-2,i-1}^{-1}\dots  \sigma_{1,2}^{-1}$.

The $i$th face operator $f_i : C^{\cs,\hFI}_p(M)_n \to C_{p-1}^{\cs,\hFI}(M)_n$ is given by $f_i = (\Ind_{\hS_{n-p}}^{\hS_n} x_{n-(p+1)}) \circ u_i$:    $$\Ind_{\hS_{n-(p+1)}}^{\hS_n} M_{n-(p+1)} \to
 \Ind_{\hS_{n-(p+1)}}^{\hS_n} M_{n-(p+1)} ~\iso~ \Ind_{\hS_{n-p}}^{\hS_n}  \Ind_{\hS_{n-(p+1)}}^{\hS_{n-p}} M_{n-(p+1)} \to
\Ind_{\hS_{n-(p+1)}}^{\hS_{n-p}}  M_{n-p}   $$These $f_i$ satisfy the semi-simplicial identites.   

We call the associated chain complex $C^{\cs,\hFI}_*(M)$ \emph{central stability chains}. We call the homology of this chain complex \emph{central stability homology} and denote it by $H^{\cs,\hFI}_*(M)$.

\end{definition}


One can define central stability chains and homology for $\FI$-modules by analogous formulas. We will use the notation $C^{\cs,\FI}_*(M)$ and $H^{\cs,\FI}_*(M)$ respectively for the central stability chains and homology of an $\FI$-module $M$. The following is \cite[Corollary 2.27]{MW1}.

\begin{proposition} \label{MWvanishing}
Let $M$ be an induced $\FI$-module with generation degree $\leq d$. Then $\left( H_i^{\cs,\FI}(M) \right )_n \leq n-2-d$. 
\end{proposition} 

We will need a slight generalization of this result to a larger class of $\FI$-modules.

\begin{definition}
An $\FI$-module is called semi-induced if it has a filtration by induced $\FI$-modules. 
\end{definition}

\begin{corollary} \label{MWvanishingCor}
Let $M$ be a semi-induced $\FI$-module with generation degree $\leq d$. Then $\left( H_i^{\cs,\FI}(M) \right )_n \leq n-2-d$. 
\end{corollary} 

\begin{proof}
By definition, semi-induced $\FI$-modules have filtrations with filtration quotients induced $\FI$-modules. Central stability chains is an exact functor. The claim follows by induction on this filtration using Proposition \ref{MWvanishing} and the long exact sequence in homology induced by a short exact sequence of chain complexes. 
\end{proof}

\begin{remark}
There are more conceptual definitions of the category $\hFI$ and its action on on  $\B_n$, and the central stability homology chain complex,  in terms of (braided) commutative monoids in a braided monoidal category. 

\end{remark}

Vanishing of central stability homology controls the generation and presentation degree.

\begin{proposition} \label{centralControlPresent}
Let $M$ be an $\hFI$-module, and let $r \geq d$. Then $H_{-1}^{\cs,\hFI}(M)_n \cong 0$ for all $n>d$ if and only if $M$ has generation degree $\leq d$. Additionally, $H_{-1}^{\cs,\hFI}(M)_n \cong 0$ for all $n>d$ and $H_{0}^{\cs,\hFI}(M)_n \cong 0$ for all $n>r$ if and only if $M$ has generation degree $\leq d$ and presentation degree $\leq r$. 
\end{proposition}

We will defer the proof until the next section.  For $\FI$-modules, a similar theorem is true. 

\begin{proposition} \label{centralControlPresentFI}
Let $M$ be an $\FI$-module. Then $H_{-1}^{\cs,\FI}(M)_n \cong 0$ for all $n>d$ if and only if $M$ has generation degree $\leq d$. Additionally, $H_{-1}^{\cs,\hFI}(M)_n \cong H_{0}^{\cs,\hFI}(M)_n \cong 0$ for all $n>r$, if and only if $M$ has presentation degree $\leq r$. 
\end{proposition} 

\begin{proof}
Let $H_i^{\FI}$ denote the $ith$ left derived functor of $H_{-1}^{\cs,\FI}$. By definition, $H_{-1}^{\cs,\FI}(M) \cong   H_{0}^{\FI}(M)$. It follows from the proof of \cite[Proposition 2.4]{CMNR} that $H_{0}^{\cs,\FI}(M)$ vanishes if and only if $  H_{1}^{\FI}(M)$ vanishes. The claim now follows from Church--Ellenberg \cite[Proposition 4.2]{CE}.
\end{proof}


\subsection{Relationship between $\hFI$ and $\FI$}

There is a functor  $\hFI \to \FI$  given by $\hS_n/\hS_{n-m} \to \rS_n/\rS_{n-m}$,  which intertwines the monoidal structures on them.  This map is an isomorphism $\hFI(n,m)  \to \FI(n,m)$ whenever  $m \geq n+2$.  We say that an $\hFI$-module $M$ is an $\FI$-module if the functor $M: \hFI \m \Mod_\Z$ factors through $\hFI \m \FI$.

\begin{proposition} 
An $\hFI$-module $M$ is an $\FI$-module  if and only if  $(e, 2)_n \in \hFI(n,n)$  acts  trivially on $M_n$ for all $n$.
\end{proposition}

\begin{proof}
Clearly, if $M$ is an $\FI$-module, then  $(e,2)_n$ acts trivially on  $\FI$.  Conversely,  we have  $\hS_n/(e,2)_n \cong  \rS_n$  and  $(e,2)_n \backslash \hFI(n,n+1) \cong \FI(n,n+1)$, so that if $(e,2)_n$ acts trivially, the action factors through $\FI(n,n+1)$.    
\end{proof}

\begin{definition}

		The monoidal structure $-\oplus - :\hFI \times \hFI \to \hFI$ gives rise to a functor $- \oplus 1 :  \hFI \to \hFI$.  We define the suspension $\Sigma M$  to be the restriction of $M$ along $- \oplus 1$. 
  The unique map $0 \to 1 \in \hFI(0,1)$, induces a natural transformation $M_{n} = M_{n \oplus 0}  \to M_{n \oplus 1} = \Sigma M$. 
\end{definition}

\begin{proposition} \label{shiftIsFI}
Let  $M$ be an $\hFI$-module generated in degree $\leq a$.  Then $\Sigma^{a +2} M$ is an $\FI$-module.
\end{proposition}
\begin{proof}
If  $n \geq a+2$,  then $(e,2)_n$  acts trivially on $M_n$,  since  $(e,2)_n$ acts trivially on $\hFI(n,m)$ for any $n \leq m-2$.  Thus $(e,2)_m$ acts trivially on $\Sigma^{a+2} M$ for all $m$.  
\end{proof}

\begin{definition}For $M$ an $\hFI$-module, we define  $\,_{\geq d} M \subset M$ to be the submodule   $$\begin{cases} M_n &  \text{ if } n \geq d  \\  0  & \text{ otherwise} \end{cases}$$   \end{definition}



We now compare $\hFI$-module central stability homology with $\FI$-module central stability homology.

\begin{proposition}\label{csComparison}
Let $N$  be an $\FI$-module. There is a canonical map of central stability complexes  $C^{\cs, \hFI}_p(N)_n \to C^{\cs, \FI}_p( N)_n$ which induces an isomorphism for  $p \leq n-1$.
\end{proposition}
\begin{proof}
When $n - (p+1) \geq 2$, the projection map $\hS_{n} \to \rS_{n}$ gives a canonical bijection $\hS_{n} /\hS_{n-(p+1)} \to \rS_{n}/\rS_{n-(p+1)}$. The projection map induces $\bbZ\hS_{n} \otimes_{\hS_{n-(p+1)} } N_n  \to  \bbZ\rS_{n} \otimes_{\rS_{n-(p+1)} } N_n$,  and writing this map in terms of coset representatives for $\hS_n/\hS_{n-(p+1)}$,  we see that it is an isomorphism. 
\end{proof}

Now we return to the proof of Proposition \ref{centralControlPresent}. The key input is the following lemma.

\begin{lemma}\label{freecomputation}
We have that $H_0^{\cs, \hFI}(\I(V)) = 0, H_{-1}^{\cs, \hFI}(\I(V)) = V$ for all $\hS$ representations $V$.
\end{lemma}
\begin{proof}
To simplify notation, we will write $H_i$ for $H_{i-1}^{\cs, \hFI}$ and $C_i$ for $C_{i-1}^{\cs,\hFI}$ throughout the proof.  

First we show the claim in the case $V = \bbZ_0$, where $\bbZ_0$ the $\hS$ representation that is $\bbZ$ in degree $0$ and $0$ elsewhere.  In this case $\I(V)_n = \bbZ\hFI(0,n)$.   We need only consider the first three terms of the complex.  Except for $ n \leq 3$, this complex agrees with the $\FI$ central stability complex of $\bbZ\FI(0,-)$,  and hence is exact by Proposition \ref{MWvanishing}. Therefore we have $H_0(\I(\bbZ_0))_n = H_1(\I(\bbZ_0))_n = 0$ for all $n \geq 3$.  It is easy to check directly that $H_0(\I(\bbZ_0))_0 = \bbZ$ and $H_1(\I(\bbZ_0))_0 = H_*(\I(\bbZ_0))_1 = 0$.  In the case $n = 2$, the complex takes the form $\bbZ \hS_2 \to \bbZ \hS_2 \to \bbZ \to 0$  where the generator $[e] \mapsto [e]-[(\sigma, 1)]$.  Using the identification $\hS_2 \cong \bbZ$,  it is clear that this complex is exact.  Finally, in the case $n = 3$,  we have $$\bbZ \hFI(2,3) \to \bbZ \hFI(1, 3) = \bbZ \FI(1,3) \to  \bbZ \hFI(0, 3) = \bbZ \FI(0,3)  \to 0. $$  Since the leftmost map factors through the surjection $\bbZ \hFI(2,3) \to \bbZ \FI(2,3)$, we may again use the exactness of the $\FI$ central stability complex of $\bbZ\FI(0,-)$ to show exactness.  This finishes the case $V= \bbZ_0$.

To extend to the general case,  notice that the complex of $\hS$ representations $C_*(\I(V))$ takes the form  $V* C_*(\I(\bbZ_0))$, where $V*M$  denotes  the induction product $\Ind_{\hS \times \hS}^{\hS} V \otimes M.$   Write $C$ for the complex $C_*(\I(\bbZ_0))$.  Using the first homology spectral sequence for $V *^{\bL} C$,  we see that $H_0(V *^{\bL} C) = H_0(V * C)$ and $H_1(V *^{\bL} C)$ has a two-step filtration with graded pieces $H_1(V*C)$ and $H_0(\Tor_1^*(V, C))$.  Here $\Tor_i^*$ denotes the $i$ derived functor for $-*-$ , which can be computed by resolving either factor.    To prove the lemma, it suffices to show that $H_0( V *^{\bL} C) = V$ and $H_1( V *^{\bL} C) = 0$.

 Using the second hyperhomology spectral sequence,  we see that $H_1(V*^{\bL} C)$ has homology bounded above by $\Tor_1^*(V, H_0(C))$ and $\Tor_0^*(V, H_1(C))$.  Both of these groups vanish:  in the first case  $H_0(C) = \bbZ_0$  and the functor $V \mapsto V * \bbZ_0 = V$  is exact, and in the second case because $H_1(C) = 0$.   Thus  $H_1(V*^{\bL} C)$ vanishes.    Similarly, the hyperhomology spectral sequence gives us $H_0(V*^{\bL} C) = V * H_0(C) = V * \bbZ_0 = V$,  completing the proof.  
\end{proof}

\begin{proof}[Proof of Proposition \ref{centralControlPresentFI}]

Again, to  simplify notation, we will write $H_i$ for $H_{i-1}^{\cs, \hFI}$ and $C_i$ for $C_{i-1}^{\cs,\hFI}$. Let $$\I(T) \to \I(W) \to \I(V) \to M \to 0 $$ be an exact sequence,  which is the beginning of a resolution of $M$ by induced modules.   Then by Lemma \ref{freecomputation} and a hyperhomology spectral sequence, we have that $H_0(M)$  and $H_1(M)$ are computed by $H_1$ and $H_0$ of the complex  $$D:= (H_0(\I(T)) \to H_0(\I(W)) \to H_0(\I(V))) = (T \to W \to V).$$  Suppose that $V_n = 0$ for all $n > d$  and $W_n \geq 0$ for all $n > r$. Then  $H_0(M)_n = 0$ for all $n > d$  and $H_1(M)_n = 0$ for all $m > r$. So generation degree $\leq d$ and presentation degree $\leq r$  imply the vanishing of homology.

To prove the converse, we use the notion of a minimal surjection.  For any $\hFI$-module $M$  we say that a surjection $\I(V) \onto M$ is \emph{minimal} if $V \subset \I(V) \to M$ is an inclusion, for all $n$ we have that $V_n = H_0(\I(V))_n \to H_0(M)_n = M_n / \hFI(n-1, n) M_{n-1}$ is a surjection, and $V_n = 0$ if $H_0(M_n) = 0$.

 For any $\hFI$-module $M$, we construct a minimal surjection $\I(V) \to M$ as follows.  Let $V_n \subset M_n$ be an $\hS_n$ subrepresentation that surjects onto $H_0(M)_n = M_n/\hFI(n,n+1) M_{n-1}$,  such that $V_n = 0$ for $n >r$.   Then the $\hS$ representation $V$  gives a map  $\I(V) \to M$. This map is surjective by the graded Nakayama lemma: clearly $\I(V)_0 = V_0 \to M_0 = H_0(M)_0$ is surjective  and inductively  $\I(V)_{n-1} \onto M_{n-1}$ implies that $\hFI(n-1,n) \I(V)_{n-1} \onto  \hFI(n-1,n) M_{n-1}$ and so  $V_n \onto M_n/ \hFI(n-1,n) M_{n-1}$ implies that $\I(V)_n \onto M_n$.   By construction it is minimal.  

Let $M$ be a module such that $H_0(M)_n = 0$ for $n > d$ and $H_1(M)_n = 0$ for $n> r \geq d$. Choose a minimal surjection $p: \I(V) \onto M$  and a minimal surjection $\I(W) \onto  K := {\rm ker} \, p$, to obtain a presentation $\I(W) \to \I(V) \to M$.   By minimality of $\I(V) \onto M$,  $V$ is nonzero only in degrees $\leq d.$  Further, the long exact sequence in homology induced by $0 \to K \to \I(V) \to M \to 0$  gives an exact sequence $$0 \to H_1(M) \to H_0(K) \to V \to H_0(M) \to 0 $$   Thus in degrees $> r \geq d$  we have that $H_0(K)$ vanishes.  By minimality of  $\I(W) \onto K$,  we have that $W_n = 0$ for $n> r$.    Thus the generation degree of $M$ is $\leq d$ and the presentation degree is $\leq r$.  \end{proof}

\subsection{Stable degree and local degree}

In this subsection, we describe how the theory of stable and local degree of $\FI$-modules can be adapted to $\hFI$-modules.

\begin{definition} \label{defLocal}
Let $M$ be an $\hFI$-module. The local degree of $M$ is the smallest number $N \geq -1$ such that $\Sigma^{N+1} M$ is a semi-induced $\FI$-module. 
\end{definition}

Following \cite{CMNR}, we denote the local degree of $M$ by $h^{max}(M)$. 

\begin{definition} \label{defStable} Let $M$ be an $\hFI$-module. Let $\Delta M$ be the cokernel of the natural map $M \m \Sigma M$. We say that $M$ is torsion if for all $n$ and all $x \in M_n$, there is an element $f \in \Hom_{\hFI}(n,m)$ with $f_*(x)=0$. The stable degree of $M$ is the smallest number $N \geq -1$ such that $\Delta^{N+1} M$ is torsion.
\end{definition}

In \cite{CMNR}, the stable degree of an $\FI$-module was defined using an analogous formula. Note that the functors $\Sigma$ and $\Delta$ in the category of $\FI$-modules defined in \cite{CMNR} agrees with their $\hFI$ analouges on the subcategory of the category of $\hFI$-modules with trivial $\Z$-action. Similarly, an $\FI$-module is torsion if and only if it is torsion when viewed as an $\hFI$-module. Thus, for $M$ an $\hFI$-module which is also an $\FI$-module, $\delta(M)$ as defined in \cite{CMNR} agrees with $\delta(M)$ as defined here. Thus, we will not distinguish between the $\FI$-module and the $\hFI$-module versions of these notions. The following follows from \cite[Proposition 2.9]{CMNR}, Proposition \ref{shiftIsFI}, and this comparison between $\FI$-module stable degree and $\hFI$-module stable degree.

\begin{proposition} \label{altDefStable} Let $M$ be an $\hFI$-module. If $\Sigma^N M$ is an $\FI$-module with generation degree $\leq d$, then $\delta(M) \leq d$.
\end{proposition}

Following \cite{CMNR}, we denote the stable degree of $M$ by $\delta(M)$. The following proposition is immediate from the definition. 

\begin{proposition} \label{shiftLocal}
Let $M$ be an $\hFI$-module. Then $h^{max}(M) \leq N+h^{max}(\Sigma^N M)$.
\end{proposition}

\begin{lemma} \label{genPlusShift}
Let $M$ be an $\hFI$-module with $\Sigma^N M$ generated in degrees $\leq d$. Then $M$ has generation degree $\leq d+N$.
\end{lemma}

\begin{proof}
Since $\Sigma^N M$ is generated in degrees $\leq d$, $$ \Ind^{\hS_n}_{\hS_d}(\Sigma^N M_{d}) \m \Sigma^N M_n  $$ is surjective for all $n \geq d$. This is equivalent to the statement that $$ \Ind^{\hS_n}_{\hS_d} M_{d+N} \m (\Sigma^N M)_{n}  $$ is surjective for all $n \geq d$. This implies that $$ \Ind^{\hS_{n+N}}_{\hS_{d+N}} M_{d+N} \m  M_{n+N}  $$ is surjective for all $n \geq d$  and so $M$ is generated in degree $\leq d+N$.
\end{proof}

The following is an adaptation of \cite[Proposition 3.1]{CMNR} to the case of $\hFI$-modules.

\begin{proposition} \label{localStableboundPresntationDeg}
Let $M$ be an $\hFI$-module with local degree $N$ and stable degree $d$. Then the generation degree of $M$ is $\leq d+N+1$ and the presentation degree is $\leq 2N+d+2$. 
\end{proposition}

\begin{proof}
Since stable degree is independent of shift, $\delta(\Sigma^{N+1} M)=d$. Since $\Sigma^{N+1} M$ is a semi-induced $\FI$-module, \cite[Proposition 2.9 (1)]{CMNR} implies that $\Sigma^{N+1} M$ has generation degree equal to $d$. By Lemma \ref{genPlusShift}, $M$ has generation degree $\leq d+N+1$. This implies $H_{-1}^{\cs,\hFI}(M)_n \cong 0$ for $n>d+N+1$ by Proposition \ref{centralControlPresent}.

Note that $\,_{\geq N+1}  M$ is an $\FI$-module. The stable degrees of $\,_{\geq N+1}  M$ and $M$ agree since they agree after sufficiently many shifts. Since $\Sigma^{N+1} \,_{\geq N+1} M  =\Sigma^{N+1} M$ is an induced $\FI$-module, $\,_{\geq N+1}  M$ also has local degree $\leq N$. By \cite[Equation ($\star $)]{CMNR}, $\,_{\geq N+1} M$ has generation degree $\leq N+d+1$ and presentation degree $\leq 2N+d+2$ as an $\FI$-module. By Proposition \ref{centralControlPresentFI}, we have that $$H_{-1}^{\cs,\FI}(\,_{\geq N+1} M)_n \cong 0 \text{ for }n>N+d+1$$ and $$H_{0}^{\cs,\FI}(\,_{\geq N+1} M)_n \cong 0\text{ for }n>2N+d+2.$$ By Proposition \ref{csComparison}, we have that $$H_{-1}^{\cs,\hFI}(\,_{\geq N+1} M)_n \cong 0 \text{ for }n>N+d+1$$ and $$H_{0}^{\cs,\hFI}(\,_{\geq N+1} M)_n \cong 0\text{ for }n>2N+d+2.$$ The natural map of $\hFI$-modules $\,_{\geq N+1} M \m M$ induces an isomorphism on $(C_p^{\cs,\hFI})_n$ for $n \geq p+N +2$. Thus, $$H_{0}^{\cs,\hFI}( M)_n \cong 0\text{ for }n>2N+d+2.$$ The claim now follows by Proposition \ref{centralControlPresent}. \end{proof}

\section{Representation stability}

In this section, we will prove that $\{H_i(F_n)\}_n$ assemble to form an $\hFI$-module which is generated in finite degree. We adapt the algebraic techniques of \cite{CMNR} to the case of $\hFI$-modules. We use connectivity results of Hatcher-Wahl \cite{HW}.

\subsection{Fundamental group of the Milnor fiber and $\hS_n$}

Our first goal is describe an action up to homotopy of $\hFI$ on the spaces $F_n$. We begin with a discussion of braid groups and fundamental groups of Milnor fibers. 

\subsubsection{Braid Conventions}\label{conventions}
Let $\Br_n$ be the braid group on $n$ strands.   We will write  $\{\sigma_{i,i+1}\}_{i = 1}^{n-1}$ for the Artin generators of $\Br_n$.   Diagramatically,  our convention is that strands are numbered $1, \dots, n$ from left to right, and we read the braids from top to bottom. The element $\sigma_{i,i+1}$ braids the $i$th strand over the $i+1$st strand.  The element $b_1 b_2$  denotes the braid $b_1$  followed by the braid $b_2$.  Similarly, given a decomposition of $[n]$ into disjoint subsets $[n] = A \sqcup B$, such that $b \geq a$ for all $a \in A, b \in B$, we let  $\sigma_{A,B} \in \rB_n$ denote the element that braids the strands of $A$ over the strands of $B$.  

We will write $\PBr_n$ for the pure braid group.  The pure braid group is generated by elements $a_{i,j}\subset \Br_n$  which braids the $i$th strand over and around the $j$th strand. That is, we have $a_{i,j} =  \sigma_{i,i+1} \dots \sigma_{j,j+1}  \sigma_{j,j+1} \sigma_{j-1,j} \inv  \dots \sigma_{i,i+1} \inv$.

The groups $\Br_n$ and $\PBr_n$ are the fundamental groups of the unordered and ordered configuration spaces of $\C$ respectively.  When we take the fundamental group of a configuration space,  we implicitly choose a base point  where all of the points are on the $x$-axis, and are in order  if the configuration space is ordered.  

Throughout, for $n \in \mathbb N$,  we define $[n] = \{1, \dots,  n\}$  to be the distinguished set with $n$ elements.  If we speak of an order on $[n]$ it will be the standard order.

Let   $q_n: \Conf_n (\bC) \to \bC^*$ be the map $q_n(x_i) =\prod_{i<j}  (x_i - x_j)$.  Then the $n$th type $A$ Milnor fiber  is $F_n = q_n \inv(1)$.  We write $\B_n := \pi_1(F_n)$ for its fundamental group. To compute the map $\pi_1 q_n : \PBr_n \to \bbZ$, notice that it factors through the abelianization,  so we may compute the map on $H_1$,  or its dual on $H^1$.   The form  $dz/z = d\log(z)$ that generates $H^1_{DR}(\bC^*)$  pulls back along $q_n$ to  $ d  \log(\prod_{i<j}  (x_i - x_j) ) = \sum_{i < j}  d \log(x_i - x_j)$.  So  the map on $H^1$  is $1 \mapsto \sum_{i,j} w_{ij}$,  where $w_{ij} \in H^1(\PBr_n, \bbZ)$ is the cohomology class that gives the winding number between two points.  Thus $\pi_1 q_n$ is given by  $a_{i,j} \mapsto 1$  for the generators of the pure braid group $a_{i,j}$  dual to $w_{ij}$.

  Since the map $q_n$ is a fibration and $\PBr_n \m \Z$ is surjective, we have a short exact sequence  $$1 \to \B_n  \to \PBr_n  \to \bbZ \to 1,$$ where $\bbZ$ is the fundamental group of $\bC^*$.  From the long exact sequence in homotopy groups associated to the fiber sequence \[F_n \m \Conf(\C) \m \bC^*, \] we see that the Milnor fiber is homotopy equivalent to the classifying space of $\B_n$.

We have the following chain of inclusions $\B_n \subset \PBr_n \subset \Br_n$.  The next proposition explains how $\hS_n$ relates to $\B_n$.  

\begin{proposition}
Let  $p_n:  \Br_n \to \hS_n$ be the map defined on generators by $\sigma_{i,i+1} \mapsto ( (i, i+1), 1)$,  which takes a braid to its associated permutation and winding number.  The map $p_n$ gives a short exact sequence:  $$ 1 \to \B_n  \to \Br_n \xrightarrow{p_n}  \hS_n \to 1.$$  
\end{proposition}
\begin{proof}
    The kernel of $p_n$ equals the kernel of its restriction to the pure braid group  $\PBr_n \to  e \times 2 \bbZ$.  This map takes the generator of the pure braid group $a_{i,j}$  to  $(e, 2)$.  Thus the map agrees with $\pi_1 p_n$, and the kernel is $\B_n$.  
\end{proof}

The category $\hFI$ acts on $F_n$ up to homotopy.  We may see this by producing an action of $\hFI$ on  $\B_n$ up to inner automorphisms.  The most direct way to define this action is as follows. 

For  $[s] \in  \hS_m/ i_2(\hS_{n-m})  = \hFI(n,m)$,  choose a lift $\tilde s \in \Br_m$ such that $p_m(\tilde s) = s$.  Then $[s] :  F_n \to F_m$  is given by    $f \mapsto \tilde s    i_1(f)  \tilde s^{-1} $.  This map is well-defined up to conjugation by elements of $\B_m$ since: \begin{enumerate}\item any other lift of $s$ differs by an element of $\B_m$, \item  if $[t]= [s]$, then $t = su$ for $u \in i_2(\hS_{n-m})$ and every lift of $u$ to an element $\tilde u$ in $i_2(\Br_{n-m})$  commutes with elements of $\B_{n}$. \end{enumerate} 
These maps compose properly  to give a functor  from $\hFI$ to the category of groups modulo inner automorphisms. Composing with the $i$th homology functor gives a functor from $\hFI$ to abelian groups. We denote this $\hFI$-module by $H_i(F)$. We use the convention that if we do not specify coefficients for homology, then the statement we make is true with any choice of untwisted coefficients.  

\subsection{Construction of a spectral sequence}

Our first goal is to construct spectral sequences, $E^*_{*,*}(N)_n$ such that the $E^2$ page is $$E^2_{p,q}(N)_n \cong  H_p^{\cs,\hFI}(\Sigma^N H_q(F))_n$$ and which converges to zero for $p + q \leq n-2$.  We do this by constructing an augmented semi-simplicial space which is highly connected.  The main input that we use is a connectivity result of Hatcher--Wahl \cite[Proposition 7.2]{HW}.

\begin{definition}
 Let  $\Br_{k,N}$ be the preimage of $i_1(\rS_k)$ under the map $\Br_{k+N} \to \rS_{k + N}$. 
\end{definition}

\begin{definition}\label{semisimplicialdef}
	Let  $Z_\bullet(N)_n$ be the following semi-simplicial set.  We let  $Z_p(N)_n =  \Br_{n,N}/i_1(\Br_{n-(p+1), N})  $, and $Z_p(N)_n = \emptyset$ for  $p \geq n$.     The $k$th face map is induced by  $$-\cdot u_k : \Br_{n,N} \to \Br_{n,N},$$  where $u_k \in \Br_{p+1}$ is as in Definition \ref{CentralStabilityDef}, and  $\Br_{p+1}$ is included into $\Br_{n,N} \subset \Br_{n+N}$  by  $$\Br_{p+1} \overset{i_2}{\into}  \Br_{n} \overset{i_1}{\into}  \Br_{n+N}.$$  
\end{definition}

In fact the complex $Z_{\bullet}(N)_n$ was studied in Hatcher--Wahl  and shown to be highly connected.
\begin{proposition}\label{HatcherWahl}
  $Z_\bullet(N)_n$ is canonically isomorphic to the semi-simplicial set  $A(D^2 - [N]; [n], [n])$  which Hatcher--Wahl \cite[Proposition 7.2]{HW} prove is $n-2$ connected.  
\end{proposition}
\begin{proof}
We will write $Z_p$ for $Z_p(N)_n$ throughout the proof. 

Elements of $\Br_{k,N}$ are isotopy classes of braids on  $k+N$ strands  which return the last $N$ braids to themselves.  Therefore,  $\Br_{k,N}$ is the fundamental group of $\Conf_{k + N}(\C)/\rS_k.$

Let $$ {\rm Arc}_{p+1}( \Conf_{n + N}( D^2) /\rS_n)$$ be the space of configurations of $n + N$ points, where the first $n$ points are unlabelled,  and $p+1$ arcs connecting a subset of the $n$ points to the boundary; see Kupers-Miller \cite[Appendix]{kupersmillercells} and Miller--Wilson \cite[Section 3.2]{MW1}.  This space is homotopy-equivalent to $\Conf_{n-(p+1) +N}(\C)/\rS_{n-(p+1)}$,  and so we have
 $$\Br_{n - (p+1), N} =  \pi_1 \left ( {\rm Arc}_{p+1}( \Conf_{n + N}( D^2) /\rS_n) \right).$$  
Further, the map $ {\rm Arc}_{p+1}( \Conf_{n + N}( D^2) /\rS_n) \to  \Conf_{n + N}( D^2) /\rS_n )$, given by forgetting the $p+1$ arcs is a fibration and on fundamental groups is given by the inclusion $\Br_{n-p+1,N} \to \Br_{n,N}$ used to defined $Z_\bullet$.  By the long exact sequence in homotopy,  
  $$  {\rm fib}( {\rm Arc}_{p+1}( \Conf_{n + N}( D^2) /\rS_n) \to  \Conf_{n + N}( D^2) /\rS_n ),$$ it is homotopy discrete and its connected components are identified with $\Br_{n,N}//\Br_{n-(p+1),N} = Z_p$.  Simultaneously, the connected components of the fiber are of isotopy classes of $p+1$ arcs from the boundary, connecting to the first $n$ of  $N + n$ points,  so that the arcs are not allowed to cross or pass through the points.  Under this identification, the face maps of $Z_p$ correspond to forgetting arcs, and so  $Z_\bullet$ is isomorphic to the complex of Hatcher and Wahl.   
\end{proof}

It is convenient to use the variant of $\Br_{k,N}$ for $\hS_n$,  and of the semi simplicial set  $Z_\bullet(N)_n$.
\begin{definition}
Let  $\hBr_{k,N}$ be the preimage of $i_1(\hS_k)$ under the map $\Br_{k+N} \to \hS_{k + N}$.

The sets  $\widehat Z_p(N)_n := \hBr_{n,N} /\hBr_{n-(p+1), N}$  form a semi-simplicial set defined by the same formulas as in Definition \ref{semisimplicialdef}.  
\end{definition}

Notice that $\hBr_{k,N} \subset \Br_{k,n}$.  In fact, the two groups are often the same.
\begin{lemma}\label{compareBkN}
     We have  $\hBr_{k,N} \cong \Br_{k,N}$ for $k \geq 2$.  
\end{lemma}
\begin{proof}
We need that the preimage of $\rS_k$ under the projection $\hS_{k+N} \to \rS_{k+N}$ is $\hS_k$.  This follows from the fact that $\hS_{k+N}/\hS_{k}  \to \rS_{k+N}/ \rS_{k}$ is an isomorphism for $k \geq 2$.  
\end{proof}

\begin{definition}
Let $\B_{n,N} \subset \hBr_{n,N}$  be the kernel of the surjection $\hBr_{n,N} \onto \hS_n$. Then  $\B_{n,N}$ acts on the semi-simplicial set $Z_p(N)_n = \hBr_{n,N} / \hBr_{n-(p+1),N}$ by left multiplication.     Let  $X_p(N)_n := \B_{n,N}\backslash \backslash Z_p(N)_n,$  where  $\B_{n,N}\backslash \backslash -$ denotes a functorial homotopy quotient.  Then  $X_\bullet(N)_n$ is a semisimplicial space,  which is augmented by the map  $X_p(N)_n \to X_{-1}(N)_n :=    \B_{n,N} \backslash \backslash *$.  

\end{definition}


 We will suppress $n,N$ from the notation for $X_p(N)_n$ when the context is clear.

\begin{proposition}\label{Compatibility}
For every $j$,  the associated  chain complex $H_j(X_\bullet(N)_n)$ is canonically isomorphic to the central stability chains  $C_\bullet^{\cs,\hFI}(\Sigma^N H_j(F_{-}))$.
\end{proposition}
\begin{proof}   
We have that \[  \B_{n,N} \backslash \backslash \hBr_{n,N} / \hBr_{n-(p+1),N}\simeq \B_{n,N} \backslash \backslash \hBr_{n,N} // \hBr_{n-(p+1),N}  \] \[   \simeq \B_{n,N} \backslash  \hBr_{n,N} // \hBr_{n-(p+1),N} \simeq \hS_{n} // \hBr_{n-(p+1),N}.\]  Thus $H_j(X_p)\iso  H_{j}(\hBr_{n-(p+1),N},  \bbZ \hS_n).$
Under this idenfication, the $i$th face operator acts by restricting $ \bbZ\hS_n$ to  $B_{n - p}$  and multiplying by  $ u_i$,  where $u_i$ is as in  \S \ref{CSHsection}.

Now the map $\hBr_{n-(p+1), N} \to \hS_n$ factors as  $\hBr_{n-(p+1), N} \onto \hS_{n-(p+1)} \into \hS_n$.  By definition of $\hBr_{n-(p+1), N}$,  the kernel of this map is $\hBr_{n -(p+1) + N} \cap  \hBr_{n-(p+1), N} = \hBr_{n -(p+1) + N}$.  

We have  $$H_a( \hS_{n- (p+1)},  H_b(\hBr_{n -(p+1) + N}, \bbZ \hS_n)) \cong H_a(\hS_{n- (p+1)}, \bbZ \hS_n \otimes H_b(\hBr_{n -(p+1) + N})$$  since the coefficients are free, this last term vanishes for $a \neq 0$,  and for $a = 0$ is equal to $\Ind_{\hS_{n- (p+1)}}^{\hS_n}(   H_b(\hBr_{n -(p+1) + N}))$.  By the Serre spectral sequence,  this shows that $H_j(X_p) \cong \Ind_{\hS_{n- (p+1)}}^{\hS_n}(   H_j(\hBr_{n -(p+1) + N}))$, as desired.  
\end{proof}

We now show that the augmented semi-simplicial space is connected in a range growing in $n$.

\begin{proposition}\label{connectivity}
 Let $n \geq 2$.  Then the augmentation map  $|X_\bullet| \to X_{-1}$  induces an isomorphism on $H_i$ for $i \leq n-2$.  
\end{proposition}
\begin{proof}

It suffices to prove that the semi-simplicial set $\widehat{Z_\bullet}$ is $n-2$ connected since it has the homotopy type of the homotopy fiber of $|X_\bullet| \to X_{-1}$. The inclusion $\hBr_{k, N} \subset \Br_{k,N}$  induces a map of simplical sets $\widehat Z \to Z$.  By Lemma \ref{compareBkN}, this map is an isomorphism on $p$ simplices for $p \leq n-2$.  Since we have assumed $n \geq 2$, it is a surjection on $n-1$ simplices.  Thus $\widehat Z$ is $n-2$ connected if $Z$ is.  Finally by Proposition \ref{HatcherWahl},  we have that $Z$ is $n-2$ connected, completing the proof.
\end{proof}

From the above two propositions, we obtain a spectral sequence with the desired properties.

\begin{proposition}
For all $n$ and $N$, there is a homologically graded spectral sequence $E^r_{p,q}(N)_n$ with $$E^2_{p,q}(N)_n \cong \left(  H_p^{\cs,\hFI}(\Sigma^N H_q(F)) \right)_n \text{ and}$$ $$E^\infty_{p,q}(N)_n \cong 0 \text{ for } p+q \leq n-2.$$  Here we take $p \geq -1$ and $q \geq 0$.

\label{SSexists}

\end{proposition}

\subsection{Proof of stability}

The following three lemmas will be used in an induction argument to prove representation stability for Milnor fibers. 

\begin{lemma} 
\label{inductionBegining}
We have that $\delta(H_0(F))  \leq 0$ and $h^{max}(H_0(F))=-1$. 
\end{lemma}

\begin{proof}
Since $H_0(F_n) \cong \Z$ and all of the stabilization maps are isomorphisms,  $\delta(H_0(F))$ is generated in degree $0$ and so $\delta(H_0(F)) \leq 0$. Since $\Sigma^0 H_0(F)$ is an induced $\FI$-module, $h^{max}(H_0(F))=-1$.


\end{proof}

The following lemma is an adaptation of the arguments in  \cite[Theorem 5.1, Part 1)]{CMNR}.

\begin{lemma} \label{deltaBound}
Let $i \geq 1$ and suppose $\delta(H_q(F))  \leq 2q$ for all $q<i$ and $h^{max}(H_q(F))$ is finite for all $q<i$. Then $\delta(H_i(F))  \leq 2i$.
\end{lemma}

\begin{proof}
Let $N$ be a number larger than $h^{max}(H_q(F))$ for all $q<i$ and take $n>2i$.  By Proposition \ref{SSexists},  we have that $$E^\infty_{-1,i}(N)_n \cong 0$$ and $$E^2_{p,q}(N)_n \cong H_{p}^{cs,\hFI}(\Sigma^N H_q(F))_n.$$ By Proposition \ref{altDefStable} and Proposition \ref{centralControlPresentFI}, it suffices to show $H_{-1}^{cs,\hFI}(\Sigma^N H_i(F))_n \cong 0$. Since $E^2_{-1,i}(N)_n \cong H_{-1}^{cs,\hFI}(\Sigma^N H_i(F))_n$ and $E^\infty_{-1,i}(N)_n \cong 0$, it suffices to show that $E^2_{-1,i}(N)_n \cong  E^\infty_{-1,i}(N)_n$. To do this, we will show that $E^2_{t,i-t}(N)_n \cong 0$ for all $t>0$. 

We have $E^2_{p,q}(N)_n \cong H^{\cs,\hFI}_p(\Sigma^N H_q(F))$. Consider $q<i$. Since $N> h^{max}(H_q(F))$, $\Sigma^N H_q(F)$ is a semi-induced $\FI$-module. Since $\Sigma^N H_q(F)$ has generation degree $\leq 2q$ and is semi-induced, Corollary \ref{MWvanishingCor} implies that $$ \left ( H^{\cs,\FI}_p(\Sigma^N H_q(F)) \right )_n \cong 0 \text{ for } p \leq n -2 - 2q.$$  By Proposition \ref{csComparison}, $$ \left ( H^{\cs,\hFI}_p(\Sigma^N H_q(F)) \right )_n \cong \left ( H^{\cs,\FI}_p(\Sigma^N H_q(F)) \right )_n \text{ for } p \leq n-1.$$ Thus $E^2_{t,i-t}(N)_n \cong 0$ for all $t>0$ and so the claim follows.
\end{proof}

The following lemma is an adaptation of the arguments in  \cite[Theorem 5.1, Part 2)]{CMNR}.

\begin{lemma} \label{lemmaLocal}
Let $i>0$ and assume $\delta(H_q(F)) \leq 2q$ for $q \leq i$ and $h^{max}(H_q(F)) \leq f(q)$ for $q<i$ for some increasing function $f$. Then $h^{max}(H_i(F)) \leq f(i-1)+6i+3$. 
\end{lemma}

\begin{proof}
Let $N=f(i-1)+1$ and let $N' \geq N$. As in the proof of Lemma  \ref{deltaBound}, we have that $$E^2_{p,q}(N')_n \cong 0 \text{ for } p \leq n -2 -2q \text{ and }q<i$$ and hence $$ H_{-1}^{\cs,\hFI}(\Sigma^{N'} H_i(F))_n  \cong 0 \text{ for }n >2i. $$ By considering $E^2_{0,i}(N')_n$ instead of $E^2_{-1,i}(N')_n$, we get the inequality $$ H_0^{\cs,\hFI}(\Sigma^{N'} H_i(F))_n  \cong 0 \text{ for } n >2i+1. $$  Let $M=\Sigma^{N+ 2i+2} H_i(F)$. By Proposition \ref{shiftIsFI}, $M$ is an $\FI$-module. By Proposition \ref{csComparison}, $H_{-1}^{\cs,\FI}(M) \cong 0$ for $n > 2i$ and $H_{0}^{\cs,\FI}(M) \cong 0$ for $n >2i+1$. By Proposition \ref{centralControlPresentFI}, we have that $M$ has generation degree $\leq 2i$ and presentation degree $\leq 2i+1$. By \cite[Equation ($\star \star$)]{CMNR}, the local degree of $M$ is $\leq 4i$. Since $M$ is an $f(i-1)+2i+3$-fold shift of $H_i(F)$, by Proposition \ref{shiftLocal}, we conclude that the local degree of $H_i(F)$ is $\leq f(i-1)+6i+3$. 
\end{proof}

Solving the recurrence and combining Lemma \ref{inductionBegining}, Lemma \ref{deltaBound}, Lemma \ref{lemmaLocal} gives the following.

\begin{proposition}
The stable degree of $H_i(F)$ is $\leq 2i$ and the local degree is $\leq -1+6i+3i^2$.
\end{proposition}

Combining this with Proposition \ref{localStableboundPresntationDeg} gives the following.

\begin{theorem} \label{mainRange}
The generation degree of $H_i(F)$ is $\leq 8i+3i^2$ and the presentation degree is $\leq 14i+6i^2$.
\end{theorem}

Since $F_n$ is an algebraic variety,  $H_i(F_n)$ is finitely generated as an abelian group for all $i$ and $n$, Theorem \ref{theoremA} follows from Theorem \ref{mainRange}. Theorem \ref{theoremB} follows from Theorem \ref{mainRange} and  Proposition \ref{shiftIsFI}.

\begin{remark}
It seems very plausible that a linear stable range is in fact optimal. Gan--Li \cite{GanLiCongruenceSubgroups} were able to prove linear stable ranges for congruence subgroups of general linear groups. Can one adapt their techniques to the case of Milnor fibers? One major obstacle to doing this is the fact that the chains on the Milnor fibers do not seem to be homotopy equivalent to an $\hFI$-chain complex. 
\end{remark}


\section{Stable calculations}
In this section,  we will study  $H_i(F_n, \Z)$ in the  range where  the action of $\mu_{n \choose 2}$  is trivial. In particular,  we will compare its homology to $H_i({\rm Conf}_n (\C) / \C^*, \Z)$  using the fact that ${\rm Conf}_n( \C) / \C^* =  F_n/\mu_{n \choose 2}$.     

\begin{theorem}\label{computation}
Suppose that $\mu_{n \choose 2}$ acts trivially  on $H_i(F_n, \Z)$ for $i \leq k$.  Then $H_i(F_n, \Z)$ is torsion free, and the map  $F_n \to {\rm Conf}_n(\C )/ \C^*$ induces an $S_n$ equivariant isomorphism on rational homology.
\end{theorem}

The main content of the above theorem is that the homology of $F_n$ is torsion free.  The rank of the group was already determined by Settepanella \cite[Theorem 1.2]{Sett}.

\subsection{Comparing $\Conf_n(\C)/\C^*$ and $F_n$}
First we note that with $\Q$ coefficients, the homology of $F_n$ is  canonically isomorphic to the homology of $\Conf_n(\C)/\C^*$.  

\begin{proposition}\label{rational}
If  $\mu_{n \choose 2}$  acts trivially on $H_i( F_n, \Q)$  for  $i \leq  k$,  then $H_i(F_n, \Q) \iso  H_i(\Conf_n(\C)/\C^*, \Q)$  for  $i \leq k$.
\end{proposition}
\begin{proof}
The group  $\mu_{n \choose 2}$ acts freely on $F_n$  and its quotient is $\Conf_n(\C)/\C^*$. Thus the rational homology of  $\Conf_n(\C)/\C^*$  is canonically  identified with the coinvariants  $H_i(F_n, \Q)_{\mu_{n \choose 2}}$  under the pushforward map.   Since $\mu_{n \choose 2}$ acts trivially for $i \leq k$,  we obtained the desired isomorphism.
\end{proof}

\begin{proposition} \label{torsionFree}
If  $\mu_{n \choose 2}$  acts trivially on $H_i( F_n, \Z)$  for all  $i \leq  k$,  then  $H_i(F_n, \Z)$ is torsion free for all $i \leq k-1$.  
\end{proposition}

To prove this proposition, we will need the following lemma. 

\begin{lemma}\label{resolution}
Let $\bbZ/m$ be an abelian group,  and let $C_*$  be a chain complex of $\bbF_p[\bbZ/m] = \bbF_p[x]/(x^{m} - 1)$   modules  concentrated in homological degree $\geq 0$.  Assume that for all $i \leq k$,  that $H_i(C_*)$ is finite dimensional, and $x - 1$ acts nilpotently on $H_i(C_*)$.  We construct a chain complex $G_*$ of projective $\bbF_p[\bbZ/m]$  modules and a quasi-isomorphism  $f: G_* \xrightarrow{\simeq} C_*$ such that:
\begin{enumerate}
\item for all  $i \leq k$,  $G_i$ is isomorphic $P^{\oplus r}$ for some $r \in \mathbb N$,  where $P$ is the module  $P := \bbF_p[x]/(x - 1)^{p^d}$, and $p^d$ is the largest power of $p$ dividing $m$,
\item for all $i \leq k +1$,  the differential $d_i: G_i \to G_{i-1}$ is zero mod $x-1$.  
\end{enumerate}
\end{lemma}
\begin{proof}
 First, note that $P$ is a summand of $$\bbF_p[x]/(x^m -1) = \prod \bbF_p[x]/(x^{q_i^{d_i}} -1)$$ by the Chinese remainder theorem, and so is projective.  Also  $x-1$  acts invertibly on all of the other factors,  so that if a power of $(x-1)$  annihilates a  element of a $\bbF_p[\bbZ/m]$-module,  then $(x-1)^{p^d}$  annihilates it.  

We construct the resolution $G_*$ inductively in the usual way. To determine $G_0$, choose $m_1,  \dots, m_{r_0}$ a collection generators of  $H_0(C_*)$ which is minimal in the sense that the associated map  $P^{\oplus r_0}  \to H_0(C_*)$  is an isomorphism mod $x-1$.   We let  $G_0 = P^{\oplus r_0}$,  and choose a lift of $G_0 \to H_0(C_*)$ to $f_0: G_0 \to  Z_0(C_*) \subset C_0$.    

To determine $G_1$,  we consider  $H_1({\rm cone}(G_0  \to  C_*)) =  \ker(G_0 \oplus C_1 \to  C_0)/d(C_2)$,  and again choose a collection of minimal generators  which give a map  $P^{r_1}  \to  H_1({\rm cone}(G_0  \to  C_*))$,  which lifts to a map $d_1 \oplus f_0:  P^{r_1}  \to  G_0 \oplus C_1$.  The map from the two term complex is now a quasi-isomorphism in degree $0$, and a surjection on homology in degree $1$.   The map  $d_1:  G_1 \to G_0$ is minimal because  its image is  $\{ g \in G_0~|~ \exists c \in C_1 , ~ f_0(g) = d(c)\}$   and we  have that for every such $g$  the homology class of $f_0(g)$ vanishes and so $g$ must be divisible by $(x-1)$  by the minimality of $f_0$. 

 To determine $G_3$,  we choose minimal generators of the second homology of the cone, $\ker(G_1 \oplus C_2 \to G_0 \oplus  C_1)/d(C_3)$,  and so on.  We continue in this way until determining  $G_{k+1}$, where we replace the role of the module $P$  by the free module $\bbF_p[x]/(x^m -1)$, and no longer require minimality of generators.
\end{proof}

\begin{proof}[Proof of Proposition \ref{torsionFree}]
  Fix $i \leq k-1$ and let $\mu = \mu_{n \choose 2}$.  We have that $H_j(F_n, \Z)$ is a finitely generated abelian group, with rank equal to the rank of the torsion free group $H_j(\Conf_n(\C)/\C^*, \Z)$ for all $j \leq k$ by Proposition \ref{rational}.  Thus by the universal coefficient theorem,  to show that  $H_i(F_n, \Z)$  is torsion free, it suffices to show that the dimension of  $H_{i+1}(F_n, \mathbb  F_p)$ equals the dimension of  $H_{i+1}(\Conf_n(\C)/\C^*, \mathbb  F_p)$  for all $p$,  $i \leq k -1$.  

  Since  $\Conf_n(\C)/\C^*$ is a quotient of  $F_n$ by a free $\mu$ action, we have  that $C_*( \Conf_n(\C)/\C^*, \bbF_p)$  is quasi-isomorphic to $G_* \otimes_{\bbF_p [\mu]} \bbF_p$ with $G_*$ any chain complex of projective $\bbF_p [\mu]$-modules quasi-isomorphic to $C_*(F_n ; \bbF_p)$. We will choose $G_*$ so it satisfies the conditions of Lemma \ref{resolution}.  

For $j \leq k$, we have  $G_{j} = P^{\oplus r_{j}}$.  By Condition 2 of Lemma \ref{resolution},  we have  $G_* \otimes_{\bbF_p [\mu]} \bbF_p$  has zero differential in degrees $\leq k$.  Thus $r_{j}$ is the dimension of $H_{j}(\Conf_n(\C)/\C^*,  \bbF_p)$.  

The dimension of $H_{j}(F_n, \bbF_p)$ is the dimension of $H_{j}(G_n)$. Call this number $c_{j}$.  We want to show that $r_{j}= c_{j}$.  We have that  $r_{j} \leq c_{j}$  because $H_j(\Conf_n(\C)/\C^*,\Z)$ is torsion free and agrees with $H_j(F_n,\Z)$ rationally  and  because the dimension of $H_j(F_n,\bbF_p)$ is at least as large as the dimension of $H_j(F_n,\Q)$.

To show that $c_{j}  \leq  r_j$,  we show that any subquotient of  $P^{\oplus r_j}$ (in particular $H_j(G_n)$)  can be generated by less than or equal to $r_j$ elements.   It suffices to show this for submodules, and every submodule of $P^{ \oplus r_j}$  pulls back to a submodule $\tilde M$ of $\bbF_p [x]^{\oplus r_j}$.   Since $\bbF_p[x]$ is a PID,  the submodule $\tilde M$ is free and thus generated by fewer than $r_j$ elements.   The images of these elements in $P^{\oplus r_j}$ generate $M$ and so  $r_j = c_j$. The claim follows.
\end{proof}

From Proposition \ref{rational} and  Proposition \ref{torsionFree}, we immediately obtain  Theorem  \ref{computation}. Theorem \ref{theoremB} and Theorem \ref{computation} imply Theorem \ref{theoremC}.


\section{Appendix}

\subsection{Central Stability Homology of Braided Monoidal Groupoids}
Let $K_n \subset \Br_n$ be a sequence of normal subgroups  such that under the map $m_{a,b}: \Br_a \times \Br_b \to \Br_{a+b}$ is contained in the image of $K_{a+b}$.  Denote the quotient by $G_n$.   Then $\{G_n\}_{n \in \mathbb N}$ forms a braided monoidal groupoid.  We have maps  $m_{a,b}: G_a \times G_b \to G_{a+b}$. The braiding is the natural transformation $m_{a,b} \to m_{b,a}$ induced by multiplication by $\sigma_{a,b} \in \Br_{a+b}$.

Write $\cA = \Rep \sqcup_n G_n$ for the category of sequences of abelian groups $A_n$  with a $G_n$ action.   The  induction product makes  $\cA$ into a braided monoidal category as follows (see e.g. \cite{joyal1993braided}).

\begin{enumerate}
\item We define $$M_ m* N_n = \Ind_{G_m \times G_n}^{G_{m+n}} M_m \otimes N_n  = \bbZ G_{m+n} \otimes_{G_n \times G_m}  M_m \otimes N_n.$$
\item Let $t_{m,n} \in \Br_{m+n}$  be the braid that passes the first $m$ strands over the last $n$.  We define the map $ t_{m,n}: M_m * N_n \to N_n * M_m$ from the action of $t_{n,m}$ on  $\bbZ G_{m+n}$  by right multiplication.
\end{enumerate}

As usual,  in a monoidal category associative algebras and modules can be defined diagramatically. From the braided monoidal structure on $\cA = \Rep \sqcup_n G_n$, we can define a \emph{commutative algebra} to be a unital associative algebra $A$,  with a multiplication $\mu : A* A \to A$  such that $\mu \circ t = t$.  \footnote{More properly, we could call $A$ an braided commutative algebra, or an $E_2$-algebra.}

Let $V$ be an object of $\cA$ and let  $\Sym_q(V) = \bigoplus_n V^{* n}/\Br_n$ be the free commutative algebra.  A right module over  $\Sym_q(V)$ consists of $M \in A$ and a map $a:  M*V \to M$, such that $a \circ ( a * \id_V) : M*V*V \to M$ equals  $a \circ ( a * \id_V) \circ(\id_M * t )$. 

 Then for any $\Sym_q(V)$-module $M$, there is a chain complex of $\Sym_q(V)$-modules $C^{\cs}_*(M)$:  $$M \leftarrow^{d_{0}}    M * V  \leftarrow^{d_1} M* V * V  \leftarrow^{d_2} M * V*V*V \leftarrow \dots, $$ defined as follows.   We have $C_{p}(M) = M * V^{* p + 1 }$ for $p \geq -1$.  The differential   $d_p = \sum_{i = 1}^{p+1} (-1)^i f_i$  is defined from an augmented semisimplicial set where the face operator $f_i :  M * V^{* p +1} \to M* V^{*p}$  acts by using the braiding to move the $i$th factor of $V$ over the other factors to $M$ and then applying the multiplication $a: M*V \to M$. 

More formally, write $u_i \in \Br_{p+1}$ for the element $\sigma_{i-1, i}^{-1} \sigma_{i-2,i-1}^{-1}\dots  \sigma_{1,2}^{-1}$ that braids the $i$th strand over the all the others to the left.   Then  $f_i$ acts by $(a * \id_V^{*p}) \circ (\id_M * u_i)$.  The semisimplicial identities hold because the multiplication map $V * V * M \to M$  factors through $ M * (V * V)/\Br_2  \to M$.  Similarly,  $M$  has the structure of a right $\Sym_q(V)$-module from by  braiding over the strands of $V$  and acting by $V$.  

 \begin{remark} The construction generalizes to produce a semisimplicial object for any object with an action of a free commutative monoid in a braided monoidal category.   \end{remark}

In the cases we consider,  $V$  is $\bbZ$, the trivial representation of $G_1$  concentrated in degree $1$. That is, we have $$V = \{V_n\}_{n \in \mathbb N} = \begin{cases}  V_1 = \mathbb Z &   \\  V_i = 0  & i \neq 1 \end{cases}.$$
Further, we will only be concerned with cases corresponding to $\hFI$ and $\FI$.  
\begin{example}
Let $G_n = \hS_n$,  and $V = \bbZ$ as above.  Then right $\Sym_q(V)$-modules are canonically equivalent to $\FI$-modules:  the data of a right $\Sym_q(V)$-module is given by maps $ M_n * \Sym_q^i(V) \to M_{n+i}$, which correspond to $$\Ind_{\hS_n \times \hS_i}^{\hS_{n+i}} M_n \otimes \bbZ \iso  \bbZ   \hS_{n+i}/i_2(\hS_i) \otimes_{\hS_{n}} M_n \to M_{n+i}.$$
  Further, we have that  $(M * V^{* p})_n =  \Ind_{i_1( \hS_{n-p})}^{\hS_n}  M_{n-p}$, and $C_*^{\cs}(M)$ agrees with $C_*^{\cs,\hFI}(M)$ as defined in Definition \ref{CentralStabilityDef}.  
\end{example}

\begin{example}
For $G_n = \rS_n$,  $\Sym_q(V)$-modules are the same as $\FI$-modules, and we obtain the $\FI$ central stability complex in the same way.  
\end{example}

For any inclusion of subgroups  $J_n \subset K_n$  with quotient  $p: H_n \onto G_n$, there is a pullback  $p^*: \Rep \sqcup_n  G_n \to  \Rep \sqcup_n  H_n$. The pullback is braided lax monoidal in the sense that there is a canonical map $p^* M * p^* N \to p^*(M * N)$,  and this  map is compatible with the braiding.

Using this structure,  $\Sym_q(V)$-modules pull back to $\Sym_q(V)$-modules.   Because central stability complexes are defined in terms of tensor powers of $V$,  the braiding, and the action of $V$ on $M$,  there is an induced map of semisimplicial complexes of $\Sym_q(V)$-modules $p^*C_*^{\cs}(M) \to C^{\cs}_* (p^* M)$.

\begin{example}
In the case of $p: \hS_n \to \rS_n$,  the map of central stability complexes agrees with the map of Proposition \ref{csComparison}.
\end{example}

\subsection{Comparison with the central stability complex of Patzt}

Let $\cC, \oplus, $  be a monoidal category such that the unit object $0 \in \cC$ is initial.   Let $\I_x: \Mod ~ \cC \to \Mod ~ \cC$ denote the left adjoint to the restriction along the functor  $- \oplus x: \cC \to \cC$.  Let $M$ be a $\cC$-module. Let $\Delta_{inj}$ denote the category governing augmented semi-simplicial objects. That is, $\Delta_{inj}$ is the category with objects finite ordered sets and morphisms given by order preserving injections. Patzt defines the \emph{central stability chains} of $M$ with respect to $x$, to be the chain complex associated to the augmented semi-simplicial abelian group  $$\Delta_{inj}^{\rm op} \to \Mod ~\cC,~ [n] \mapsto \I_{x^{\oplus n}} M.$$ 
For an ordered injection $f: [n] \to [m] \in \Delta_{inj}([n],[m])$, the associated map  $\I_{x^{\oplus n}} \leftarrow \I_{x^{\oplus m}}$ is adjoint to the natural transformation $\bS_{x^{\oplus n}} \to \bS_{x^{\oplus m}}$ induced by the morphism $f: x^{\oplus n} \to x^{\oplus m}$.  

In our setting $\cC$ is the category $\hFI$, and $x$ is the object $1$.  To compute the functor $\I_{1}^{\hS}$  in this case,  we note that there are restriction and induction functors $\bS_{1}^{\hS}$ and $\I_{1}^{\hS}$ defined on the category $\Mod \, \hS$.  

 In fact, when $M$ is an $\hFI$-module  $\I_1^{\hS} M$ carries a canonical $\hFI$-module structure.  To see this, we identify $\hFI$-modules with $\Sym_q(V)$-modules, where $V$ is the $\hS$ representation consisting of $V$ concentrated in degree $1$. Observe that  $\I_1^{\hS} M =  M*V$, where   $*$ denotes the induction tensor product  of $\hS$-modules.   Then $M*V$ becomes a right $\Sym_q(V)$-module through the map $$M* V*\Sym_q(V) \to^{\id_M * t}  M * \Sym_q(V) * V \to^{a *\id_V}  M* v,$$ where $t$ denotes the braiding, and $a$ denotes the action map for the $\hFI$-modules structure on $M$.   

This lifts $\I_1^{\hS}$ to a functor $\Mod ~\hFI \to \Mod ~\hFI$,  and we have that $\I_1^{\hS} \simeq \I_1^{\hFI}$ is adjoint to $\bS_1^{\hFI}$. In other words, let $M, N$ be $\hFI$-modules.  Then a map of $\hS$-modules $\I^{\hS}_1 M \to N$ is a map of $\FI$-modules if and only if $M \to \bS_1^{\hS} N = \bS_1^{\hFI} M$ is a map of $\hFI$-modules.  

  Under this identification, the central stability complex of Patzt corresponds to the central stability complex of \S \ref{CSHsection}.  The presence of a braiding in the differentials of our central stability complex corresponds to the braiding used to define the $\hFI$-module structure of $\I_1^{\hS} M$, and thus the corresponding adjoint maps  $\I_{1^{\oplus r}}^{\hS} M \to \I_{1^{\oplus s}}^{\hS} M $.  

\subsection{Combinatorial description of the homology of $\Conf_n(\C)/ \C^*$}

 There is a well known homeomorphism  $\Conf_n(\C)/ (\C^* \ltimes \C) \iso \Conf_n(\mathbb P^1)/{\rm PGL}_2 = M_{0,n+1}.$  Because $\C$ is contractible, this gives an isomorphism  $H_i(\Conf_n(\C)/\C^*, \Z) = H_i(M_{0,n+1}, \Z)$.  These homology groups were first computed by Getzler \cite{GetzlerModuli}.    These groups have also appeared in the representation stability literature in the work of Hyde--Lagarias \cite{hyde2017polynomial}.

In this appendix, we describe $H_d(\Conf_n(\C)/\C^*, \Z)$ in terms of the combinatorics of the $\Lie$ representations. The description we give here is based on a resolution also due to Getzler  \cite{Getzler}.  

Let  $\Lie(n)$ be the integral Lie representation of $\rS_n$. For any set $S$, we define $\Lie(S)$ to be the free abelian group on all bracketings of the elements of $S$ modulo insertions of the  anticommutativity and the Jacobi relations.

\begin{example}  As an $\rS_3$ representation,  may write a presentation of   $\Lie(\{1,2,3\})$ as $$ \Lie(\{1,2,3\}) =\frac{ \bbZ \bS_3 \{  [[ 1  2]  3],  [1[ 2  3]]   \}}{  [ [ 1  2]  3]  = - [[  2  1]  3],~  [[ 1  2]  3]   = - [ 3 [ 1  2]],~ [1[23]] + [3[12]] + [2[31]] = 0} $$
\end{example}

We write  $\Lie^\vee(n)$ for  $\Hom(\Lie(n)\otimes \sgn, \Z)$, the dual of the twist by $\sgn$. Let $P(n)$ be the poset of set partitions of $n$, ordered by refinement so that the indiscrete partition is smallest and the discrete partition is largest.   If $p \in P(n)$ is given by $[n] = b_1 \sqcup \dots \sqcup b_r$,  then  we will refer to the $b_i$  as the blocks of $p$  and  to $c(p) = n-r$  as the codimension of $p$.   We define  $\Lie^\vee(p) = \bigotimes_{i = 1}^r \Lie^\vee ( b_i).$

\begin{definition}Let $$C_k =\bigoplus_{p \in P(n), c(p) = k} \Lie^\vee(q)$$  We define a map  $d_k:  C_{k+1} \to C_k$,
in terms of the matrix coefficients $d_k(q,p):  \Lie^\vee(q) \to \Lie^\vee(p)$, as follows: \begin{itemize}
\item If $q , p$ are incomparable,  then $d_k(q,p) := 0$.  
\item If $q \leq p$ then there are no intermediate partitions,  so there exist unique blocks  $b_i, b_j$ of $p$ such that $q$ is given by $(b_i \cup b_j)  \sqcup \dots  \sqcup b_r$.  There is a map  $$ \Lie(b_i) \otimes \Lie(b_j)  \to  \Lie(b_{i+j})$$  given by inserting the brackets of the elements $b_i$ and $b_j$ into the bracket $[12]$.   If we dualize and twist by signs,  we obtain a map  $$m_{i,j}: \sgn(b_{i+j}) \otimes \Lie^\vee(b_{i+j}) \to \Lie^\vee(b_i) \otimes \Lie^\vee (b_j),$$ which is independent of the labelings of the blocks.  We define  $d_k(q,p) := ( \otimes_{r \neq i,j} \id_{\Lie^{\vee}(b_r)}) \otimes m_{i,j}$.  
\end{itemize}
 \end{definition}

Notice that $C_k$ is an $\rS_n$ representation,  a permutation $\sigma$ acts by taking the summand $$\Lie^\vee(b_1) \otimes \dots \otimes \Lie^\vee(b_r)$$  to  $\Lie^\vee(\sigma(b_1)) \otimes \dots \otimes \Lie^\vee(\sigma(b_r))$  by the map $\Lie^\vee(\sigma|_{b_i})$.

 \begin{theorem}[Getlzer]
The map $d_k:  C_{k+1} \to C_{k}$  makes  $C_*$ into an exact chain complex of $\rS_n$ representations.  Further,  $\coker(d_k^\vee) \iso  H_k(\Conf_n(\C)/\C^*,\Z)$ as $\rS_n$ representations.
\end{theorem}
\begin{proof}
Getzler \cite[Sec 1.17]{Getzler} shows that $(C_* d_*)$ is exact, and identifies $d_k$ it with the action of the fundamental class, $\epsilon$,  of $\mathbb C^* \simeq S^1$ on $H^*(\Conf_n (\C))$. This suffices to determine the integral cohomology of the quotient $\rS_n$ equivariantly. One method of computation is as follows.  The $E^2$ page of the Moore spectral sequence $\Tor_*^{H_*(\mathbb C^*)}(\bbZ, H_*(\Conf_n(\mathbb C)))$ \cite[Theorem 7.28]{mccleary2001user} can be computed using $H_*(\mathbb C^*) \iso \bbZ [ \epsilon]/\epsilon^2$ and the minimal resolution of $\bbZ$ over this ring. In particular, the $E^2$-page is a direct sum of shifts of truncations of the complex $C_*^\vee.$  This, together with Getzler's exactness, shows that the spectral sequence degenerates at $E^2$ and $H_k$ is isomorphic to $\coker \, d_k$.  
\end{proof}

As a consequence of the exactness of $C_k$,  we can compute the rank of $\coker(d_k^\vee)$, as well as the $\rS_n$ character as an alternating sum.

\begin{corollary}
 The group $H_d(\Conf_n(\C)/\C^*, \Z)$ is  free abelian of rank $$r_d = (-1)^d \sum_{i \leq d}  s(n,n-i), $$ where $s(n,k)$ denotes the signed Stirling number of the first kind.   As $\rS_n$ representations we have $${\rm ch}({H_d(\Conf_n(\C)/\C^*,\Q)}) = (-1)^d \sum_{k \leq d} (-1)^k {\rm ch}(C_k),$$ where ${\rm ch}$ denotes the Frobenius character.
\end{corollary}

 We can describe the representation $C_k$ as follows.

\begin{proposition} The frobenius character of $C_k$ is given by the symmetric function  $${\rm ch}(C_k) = \sum_{1^{m_1} 2^{m_2} \dots  \text{ integer partition of } n, ~\sum_{i} (i - 1) m_i = k } \prod_{i} h_{m_i}[{\rm lie}^{\vee}_i],$$ where $h_{m}$ is the homogeneous symmetric function,  ${\rm lie}^{\vee}_i$ is the symmetric function corresponding to the representation $\Lie^\vee(i)$,  and  $f[g]$ denotes plethysm of symmetric functions.  \end{proposition}
\begin{proof}
	Decompose $C_k$ into summands $C_k(\lambda)$  consisting of all $\Lie^\vee(p)$  where $p \in P(n)$ has type $\lambda$.  We have $$C_k(1^{m_1} 2^{m_2} \dots) = \bigoplus_{[n] = a_1 \sqcup a_2 \sqcup \dots  } \bigotimes_{i \geq 1} \bigoplus_{a_i =  t_1 \sqcup \dots \sqcup t_{i/n}, |t_i| = i } \Lie^\vee(t_1) \otimes \dots \otimes \Lie^\vee(t_i),$$ where the first sum is over \emph{ordered} partitions of $n$, and the second is over unordered partitions of $a_i$ into blocks of size $i$.   This representation is an induction product of plethysms with the trivial representation, and hence has Frobenius character   $ \prod_{i} h_{m_i}[{\rm lie}^{\vee}_i]$.   
\end{proof}

There are formula that express the symmetric functions ${\rm lie}^{\vee}$ in terms of power sums and mobius numbers (of $\bbZ$), which can be used to make the above formula more explicit.   

\begin{theorem}[Stanley, \cite{stanley1982some}]
	We have ${lie}^{\vee} = \frac{(-1)^n}{n} \sum_{d|n} \mu(d) (-1)^{n/d} p_{d}^{n/d}$.
\end{theorem}


\bibliographystyle{amsalpha}
\bibliography{MilnorFiber}

\end{document}